\documentclass[12pt]{article}
\usepackage[utf8]{inputenc}
\usepackage{amsmath}
\newcommand{\N}{\mathbb{N}}
\newcommand{\R}{\mathbb{R}}
\newcommand{\Z}{\mathbb{Z}}
\newcommand{\Q}{\mathbb{Q}}

\newcommand{\e}{\varepsilon}
\usepackage{relsize}
\usepackage{amsthm}
\usepackage{amssymb}
\usepackage[margin=1in]{geometry}
\newtheorem{lemma}{Lemma}

\newtheorem*{question*}{Question}
\newtheorem{theorem}{Theorem}
\newtheorem{corollary}{Corollary}
\numberwithin{equation}{section}
\numberwithin{lemma}{section}
\numberwithin{proposition}{section}
\numberwithin{corollary}{section}
\numberwithin{theorem}{section}
\usepackage{ccfonts}
\usepackage[T1]{fontenc}
\DeclareFontSeriesDefault[rm]{bf}{sbc}
\usepackage{bm}
\usepackage[parfill]{parskip}
\usepackage{titlesec}
\usepackage{stmaryrd}
\titleformat{\section}{\centering\normalfont\scshape}{\thesection}{1em}{}
\titleformat{\subsection}{\centering\normalfont\scshape}{\thesubsection}{1em}{}

\usepackage{hyperref}
\hypersetup{
    colorlinks=true,
    linkcolor=blue,
    filecolor=magenta,      
    urlcolor=cyan,
    pdftitle={Overleaf Example},
    pdfpagemode=FullScreen,
    }
\title{\Large Pseudo-Effectivity in the Skolem--Mahler--Lech Theorem}
\author{Victor Shirandami\footnote{Institut de Recherche en Informatique Fondamentale (IRIF)\\ Email: shirandami@irif.fr}}
\date{}

\begin{document}

\maketitle
\newcommand{\Gm}{\mathbb{G}_m}
\renewcommand{\L}{\mathcal{L}}
\begin{abstract}
The Skolem--Mahler--Lech Theorem asserts that a non-degenerate linear recurrence sequence (LRS) has only finitely many zeros, but is ineffective in the sense that no general procedure is known to determine them. In this paper a \emph{pseudo-effective} approach is introduced: rather than seeking effective bounds for every sequence, one obtains explicit bounds valid for all but a controlled number of sequences of bounded height, up to a natural equivalence. Such results are established for order $3$ LRS's, and extended to all higher orders subject to suitable height constraints on the characteristic roots. The underlying mechanism is a quantitative theory for exponential-polynomial equations subject to low-height multiplicative perturbations. In particular, it is shown that the finiteness phenomenon underlying the Skolem--Mahler--Lech Theorem persists under arithmetic perturbations whose height is permitted to grow linearly at a sufficiently small rate. As a further application, a moving-target finiteness result is established for linear orbits entering expanding arithmetic neighbourhoods of proper linear subspaces.

\end{abstract}

\section{Introduction}\label{introduction}
The celebrated Skolem--Mahler--Lech Theorem determines the structure of the set of zeros of a linear recurrence sequence (LRS) over a field of characteristic zero: it is a finite union of arithmetic progressions together with a finite set. Recall that an LRS is a sequence $(a_n)_{n \geq 0}$ subject to a relation of the form
\begin{equation}\label{LRS}
    \forall \; n \geq r: \quad a_n = u_1 a_{n-1} + \dots + u_r a_{n-r},
\end{equation}
with $u_1,...,u_r$ some elements of a fixed field of characteristic zero. At the heart of the theory of linear recurrence sequences is the presence of \emph{exponential polynomials}. These are expressions of the form
\begin{equation}\label{exponential polynomial}
    c_1(n) \alpha_1^n + \dots + c_k(n) \alpha_k^n
\end{equation}
where each coefficient $c_1(n),\dots, c_k (n)$ is a polynomial function. The connection between LRS's and exponential polynomials is realised by ascribing to the linear recurrence relation (\ref{LRS}) its corresponding \emph{characteristic polynomial}
\begin{equation*}
    f(X) := X^r - u_1 X^{r-1} - \dots - u_r.
\end{equation*}
From this one finds a representation of the LRS $(a_n)_{n \geq 0}$ in the form (\ref{exponential polynomial}) where $\alpha_1,\dots,\alpha_k$ are the roots of $f(X)$ and the degree of the $j$-th polynomial coefficient $c_j(n)$ is precisely the multiplicity of the corresponding characteristic root $\alpha_j$ diminished by $1$. Moreover, further properties of the characteristic polynomial of $(a_n)_{n \geq 0}$ serve to describe the properties of its zero set. In particular, a refinement of the Skolem--Mahler--Lech Theorem asserts that when an LRS does not possess a root of unity among the ratios of its distinct characteristic roots, then its zero set is finite. This property of the characteristic roots is referred to as \emph{non-degeneracy}, and it is known that any LRS may be effectively decomposed into subsequences of the form $(a_{Mn + m})_{n \geq 0}$ (for $0 \leq m < M$) such that each subsequence is either identically zero or non-degenerate (cf. \cite{recurrencesequences,BerstelMignotte1976,Robba1978}). Therefore, for the purpose of studying the zeros of LRS's it suffices to investigate the zeros of exponential polynomials, as in (\ref{exponential polynomial}), under the assumption that no ratio $\alpha_i/\alpha_j$ is a root of unity for $1 \leq i < j \leq k$. In the particular instance where the characteristic roots are all simple and non-zero, one may recast the vanishing of (\ref{exponential polynomial}) in affine form:
\begin{equation}\label{affine equation}
    b_1 \beta_1^n + \dots + b_{s} \beta_{s}^n = 1
\end{equation}
with $s:= k-1$, $b_j := -c_j/c_k$, and $\beta_j := \alpha_j/\alpha_k$ for each $1 \leq j \leq s$. Equivalently, a more geometrical formulation reads:
\begin{equation*}
    (b_1 \beta_1^n, \dots, b_s \beta_s^n) \in \Pi
\end{equation*}
where 
\begin{equation}\label{Pi}
    \Pi := \big\{ \bm x \in (\overline \Q^*)^s : x_1 + \dots + x_s = 1 \big\}.
\end{equation}
A conspicuous feature of the Skolem--Mahler--Lech Theorem is its ineffectivity: there is not known an effective procedure to determine the finitely many zeros of a given non-degenerate LRS. Indeed the original proofs by Skolem, Mahler, and Lech \cite{Skolem1934,Mahler1935,Lech1953} all make use of $p$-adic analysis, offering no path back towards effective bounds. In the direction of obtaining quantitative bounds on the \emph{number} of zeros a linear recurrence sequence may have, substantial progress has been made, and in fact there exist multiplicity bounds depending only on the order of the LRS. The first of this kind is due to Schmidt \cite{Schmidt1999,Schmidt2000} and refined further in \cite{Evertse2002,AmorosoViada2011}. These methods involve the Subspace Theorem (cf. \cite[Chap. 7]{BombieriGubler2007}) and constitute alternative proofs of the Skolem--Mahler--Lech Theorem, but suffer still from the ineffectivity present in the Subspace Theorem. On the other hand, genuinely effective results have been obtained, via Baker's Theory of linear forms in logarithms of algebraic numbers (cf. \cite{Bugeaud2018}), in the low order cases. Most notably for algebraic LRS's of order $3$ and real algebraic LRS's of order $4$ effectivity was established in \cite{Mignotte1975,Shorey1984,Vereshchagin1985}, and then very recently the general algebraic order $4$ setting was settled by Bacik \cite{Bacik2025}. However, a general effective theory does not at present extend past the order $4$ context.

\subsection{The Pseudo-Effectivity Problem}

The absence of effective results beyond order 4 motivates a weaker question: rather than seeking an effective bound for every LRS, can one obtain explicit bounds valid for all but a controlled number of sequences? In broad terms, this question may be formulated as follows:

\begin{question*}[Pseudo-Effectivity Problem]
    How many algebraic LRS's of bounded height, subject to a fixed recurrence relation, admit a zero with index at least a given threshold $\nu \in \N_0$?
\end{question*}

The principal aim of this paper is to establish results of this type, which are referred to here as \emph{pseudo-effective}. This is done by developing a general quantitative result concerning low-height multiplicative perturbations of exponential-polynomial equations. As applications, pseudo-effective results for order $3$ LRS's are obtained, the framework is extended to higher-order recurrences satisfying suitable height conditions on their characteristic roots, and moving-target results for linear orbits are derived. Before developing the general theory, a particularly concrete consequence is recorded here in the order $3$ setting, where surprisingly small quantitative bounds already arise. Indeed, fix $\alpha_1, \alpha_2, \alpha_3 \in \overline \Q^*$, and with $h: \overline \Q^* \to \R_{\geq 0}$ denoting the logarithmic Weil height function (cf. \cite{BombieriGubler2007}), fix
\begin{equation*}
    \xi := \max_{1 \leq i < j \leq 3} h(\alpha_i/\alpha_j),
\end{equation*}
subject to $\xi > 0$. Consider now sequences of the form
\begin{equation}\label{simple sequences}
   a_n =  c_1 \alpha_1^n + c_2 \alpha_2^n + c_3 \alpha_3^n \quad \quad (n \in \N_0),
\end{equation}
where the coefficients  $c_1, c_2, c_3 \in \overline \Q^*$ are permitted to vary subject to
\begin{equation}
    \max_{1 \leq i < j \leq 3} h(c_i/c_j) \leq \theta,
\end{equation}
for a fixed $\theta \geq 0$. It is clear that any sequence of the form (\ref{simple sequences}) is an order $3$ LRS with simple characteristic roots. The zero sets of the sequences (\ref{simple sequences}) are not entirely independent: translating the index and rescaling the sequence by a non-zero constant produces a sequence whose zero set is a translate of the original. Such sequences are regarded as equivalent in what follows. With the above notation, a first pseudo-effective result may now be stated.
\begin{theorem}\label{Order 3 Theorem simplified}
    At most $24$ inequivalent sequences $(a_n)_{n \geq 0}$ as above possess a zero with
     \begin{equation*}
         n \geq \frac{2.8 + 6\theta}{\xi},
     \end{equation*}
     while at most $5$ possess a zero with
     \begin{equation*}
         n \geq \frac{33.1 + 73.1 \theta}{\xi} \cdot
     \end{equation*}
\end{theorem}
Results of this type persist at arbitrarily high orders. To illustrate this concretely, consider the following recurrence relation of order $10$:
\begin{equation}\label{high order example}
\begin{split}
&20a_n+28a_{n-1}+41a_{n-2}+46a_{n-3}+52a_{n-4}+55a_{n-5}\\ 
&+52a_{n-6}+46a_{n-7}+41a_{n-8}+28a_{n-9}+20a_{n-10}=0.
\end{split}
\end{equation}
This example is chosen because all ten of its characteristic roots have complex modulus $1$, while eight of them form a sufficiently low-height arithmetic cluster relative to a distinguished pair. The notion of a low-height cluster is made precise in \S\ref{pseudo-effectivity at higher orders}. Although sequences satisfying this recurrence lie far beyond the range for which general effective bounds are presently known, they remain accessible to the pseudo-effective theory developed below. Writing the characteristic roots as $\alpha_1,\ldots,\alpha_{10}$, consider exponential polynomials
\begin{equation}\label{order10}
a_n=c_1\alpha_1^n+\cdots+c_{10}\alpha_{10}^n
\end{equation}
whose coefficients satisfy
\begin{equation}\label{coefficient height order 10}
\max_{1\leq i<j\leq 10}h(c_i/c_j)\leq\theta.
\end{equation}
At higher orders, a coarser notion of equivalence is used, determined only by the coefficients corresponding to a distinguished pair of characteristic roots; the precise definition is given in \S \ref{pseudo-effectivity at higher orders}. With respect to this equivalence, for every $X\geq11$ there is an explicit threshold $\nu$ such that at most $X$ inequivalent sequences of the form (\ref{order10}) possess a zero with index $n\geq\nu$. For the relatively generous coefficient-height bound $\theta=10$, some resulting values are displayed below.

\begin{table}[ht]
    \centering
    \begin{tabular}{c|ccccccc}
        $X$   & $11$ & $16$ & $21$ & $26$ & $31$ & $36$ & $41$ \\
        \hline
        $\nu$ & $761$ & $127$ & $90$ & $79$ & $74$ & $72$ & $71$
    \end{tabular}
    \caption{\small{Explicit thresholds \(\nu\) for the order \(10\) recurrence (\ref{high order example}), such that at most \(X\) inequivalent sequences satisfying (\ref{coefficient height order 10}) with $\theta = 10$ possess a zero with \(n\ge\nu\).}}
    \label{tab:X-nu}
\end{table}

To put the size of the coefficient constraint into perspective, in terms of the multiplicative Weil height $H:=\exp h$, the condition \(\theta=10\) permits coefficient ratios of multiplicative height as large as
\begin{equation*}
H(c_i/c_j)\leq e^{10}\approx 2.2\times10^4.
\end{equation*}
For coefficients of smaller arithmetic complexity the bounds improve considerably, as shown in Table \ref{tab2:X-nu} for the case where $\theta = 1$. Here, one obtains $(X,\nu) = (36,10)$, which implies that every sequence of the form (\ref{order10}) satisfying (\ref{coefficient height order 10}) with $\theta = 1$ has at most $36$ zeros with index $n \geq 10$. Thus, even at order \(10\), one obtains a completely explicit and rather small bound on the number of large zeros. The arithmetic structure of the characteristic roots which makes these estimates possible, together with the derivation of the displayed numerical bounds, is given in \S \ref{pseudo-effectivity at higher orders}.

\begin{table}[ht]
    \centering
    \begin{tabular}{c|ccccccc}
        $X$   & $11$ & $16$ & $21$ & $26$ & $31$ & $36$ & $41$ \\
        \hline
        $\nu$ & $98$ & $17$ & $12$ & $11$ & $11$ & $10$ & $10$
    \end{tabular}
    \caption{\small{Explicit thresholds \(\nu\) for the order \(10\) recurrence (\ref{high order example}), such that at most \(X\) inequivalent sequences satisfying (\ref{coefficient height order 10}) with $\theta = 1$ possess a zero with \(n\ge\nu\).}}
    \label{tab2:X-nu}
\end{table}

The above results are the simplest manifestations of the pseudo-effective theory developed in \S \ref{pseudo-effective bounds in skolem's problem}, whose technical core is a general theory of zeros of perturbations of exponential polynomials, to which the remainder of this introduction is dedicated.

The starting point for this theory is a broad generalisation of equations of the form (\ref{affine equation}). Namely, the summands $b_j\beta_j^n$ are allowed multiplicative perturbations of low height. To make this notion precise it is convenient to pass to the algebraic torus $\Gm^s(\overline{\Q})$, which is the abelian group given by the set $(\overline{\Q}^*)^s$ endowed with the product
\begin{equation*}
    \forall \; \bm x = (x_1,\dots,x_s), \bm y = (y_1,\dots,y_s) \in \Gm^s(\overline \Q): \quad \bm x \bm y := (x_1 y_1, \dots, x_s y_s) \in \Gm^s(\overline \Q). 
\end{equation*}
As a particular instance, one can write $\bm b :=(b_1,\dots,b_s)$ and $\bm \beta := (\beta_1,\dots,\beta_s)$ as elements of $\Gm^s(\overline \Q)$ such that the equation (\ref{affine equation}) becomes
\begin{equation}\label{group version of affine equation}
    \bm b \bm \beta^n \in \Pi,
\end{equation}
with $\Pi$ as defined in (\ref{Pi}). Since all objects in this work are taken over the field of algebraic numbers $\overline \Q$, in all that follows $\Gm^s(\overline \Q)$ is abbreviated to $\Gm^s$. As well as a group structure, $\Gm^s$ admits natural notions of height and distance. Indeed, with $h: \overline \Q^* \to \R_{\geq 0}$ denoting the logarithmic Weil height function as before, it may be applied to $\Gm^s$ via the formula
\begin{equation}\label{affine height definition}
   \forall \; \bm x \in \Gm^s: \quad  h(\bm x) := \max_{1 \leq j \leq s} h(x_j).
\end{equation}
This in turn is used to induce a distance function on $\Gm^s$ via the group product, given by
\begin{equation}\label{affine distance definiton}
    \forall \; \bm x, \bm y \in \Gm^s: \quad d(\bm x, \bm y) := h(\bm x \bm y^{-1}).
\end{equation}
This is easily verified to be a pseudo-metric on $\Gm^s$, and is moreover a genuine metric on $\Gm^s/\mu_\infty^s$, where $\mu_\infty \leq \overline \Q^*$ denotes the subgroup of $\overline \Q^*$ consisting of all the roots of unity (or equivalently, $\mu_\infty^s$ is the torsion subgroup of $\Gm^s$). A notable use of this metric structure (for $s=1)$ is in the Allcock-Vaaler space \cite{AllcockVaaler2009} which realises the completion of $\overline \Q^*/\mu_\infty$ with respect to this metric as a Banach space over $\R$. With the above notation, equation (\ref{affine equation}) exhibits a natural generalisation:
\begin{equation}\label{generalised equation}
    \bm F(n) \in \Pi
\end{equation}
for functions $\bm F : \N_0 \to \Gm^s$ which are pointwise close to the function $n \mapsto \bm \beta^n$ with respect to the pseudo-metric $d$. The following theorem establishes a finiteness statement for the generalised exponential-polynomial equation (\ref{generalised equation}) in the low-degree case $s=2$, while permitting the arithmetic distance $d(\bm F(n),\bm \beta^n)$ to diverge to infinity. In fact, this distance is allowed to grow linearly.

\begin{theorem}\label{main theorem soft}
    Let $\bm \beta \in \Gm^2$ be non-torsion, and let $\bm F : \N_0 \to \Gm^2$ be a function satisfying
    \begin{equation*}
        \limsup_{n \to \infty} \frac{1}{n} d\big(\bm F(n), \bm \beta^n\big) < \frac{1}{9} h(\bm \beta).
    \end{equation*}
    Then
    \begin{equation*}
        \#\big( \operatorname{im}(\bm F) \cap \Pi \big) < \infty.
    \end{equation*}
\end{theorem}
Theorem \ref{main theorem soft} may be viewed as a stability result for the finiteness phenomenon underlying the Skolem--Mahler--Lech Theorem. Indeed, whereas the classical setting concerns intersections of the exponential orbit $n \mapsto \bm b\bm\beta^n$ with $\Pi$, the above theorem shows that finiteness persists under arithmetic perturbations of this orbit whose height is permitted to grow linearly, at a sufficiently small rate relative to $h(\bm\beta)$. A quantitative form of this theorem is also obtained, leading to the explicit bounds in Theorem \ref{Order 3 Theorem simplified}. Due to the technicality of its statement, the quantitative form of Theorem \ref{main theorem soft} is postponed to \S \ref{low height perturbations of exponential polynomials}. The perturbative framework also admits applications beyond the pseudo-effective theory of Skolem's Problem. As an example, \S \ref{linear orbits} applies the quantitative theory to a moving-target problem in linear dynamics, establishing finiteness for linear orbits entering expanding arithmetic neighbourhoods of proper linear subspaces.

The methods used to prove Theorem \ref{main theorem soft}, and its quantitative counterpart Theorem \ref{main theorem hard}, have their roots in Thue's Method, originally devised to bound the irrationality measure of real algebraic irrational numbers \cite{thue1909}. In the work of Beukers \cite{Beukers1997} and Beukers \& Tijdeman \cite{beukersTijdemann1984}, auxiliary polynomials are explicitly constructed to obtain upper bounds for the heights of solutions of equations of the form
\begin{equation*}
uX^n + v(1-X)^n = 1
\end{equation*}
for some $u,v \in \overline \Q^*$, $n \in \N_{\geq 2}$ (see also \cite[Chap. 17]{Masser2016} for an accessible account of this equation). These methods are applied in \cite{beukersTijdemann1984, BeukersEvertse1996} to multiplicity estimates for zeros of order $3$ LRS's, while a broad generalisation of these techniques appears in \cite{Amoroso2017} in the study of parametric families of LRS's (cf. \cite{OstafeShparlinski2022}). The proof of Theorem \ref{main theorem hard} adapts these auxiliary-polynomial estimates to generalised exponential-polynomial expressions whose coefficients are subject to arithmetic perturbations.

\paragraph{Organisation of the Paper.}
The pseudo-effective theory of Skolem's Problem is developed in §\ref{pseudo-effective bounds in skolem's problem}, with the application to linear dynamics treated separately in §\ref{linear orbits}. Finally, §\ref{proof of the main theorem} is dedicated to the proof of Theorem \ref{main theorem hard}.

\paragraph{Acknowledgements.}
The author is grateful to Faustin Adiceam for inspiring the pseudo-effective approach developed in this paper, and to Valérie Berthé and Faustin Adiceam for their helpful remarks, which led to substantial improvements in the final manuscript. This work was supported by the European Research Council through ERC grant DynAMiCs (101167561).

\section{Pseudo-Effective Bounds in Skolem's Problem}\label{pseudo-effective bounds in skolem's problem}
As discussed in the Introduction, the results of this section are obtained as applications of a quantitative refinement of Theorem \ref{main theorem soft}. In \S \ref{low height perturbations of exponential polynomials} this refinement is stated, and in \S \ref{projective notations} the height and arithmetic distance functions are reformulated projectively to better suit the homogeneous exponential polynomials arising from linear recurrence sequences. Then, in \S \ref{pseudo-effectivity at order 3} and \S \ref{pseudo-effectivity at higher orders} the main pseudo-effective results are presented.

\subsection{Low Height Perturbations of Exponential Polynomials}\label{low height perturbations of exponential polynomials}
To state the quantitative theory concerning zeros of low-height perturbations of exponential polynomials, several auxiliary quantities must first be defined. First, given $\delta, \theta, X > 0$ satisfying
\begin{equation}\label{parameter constraints}
    \delta \in \left(0, \frac{1}{9}\right) \quad \text{and} \quad X > \frac{\log \left( \frac{6-3\delta}{1-9\delta} + \frac{1}{2} \right)}{\log \left(\frac{3-2\delta}{2(1+\delta)}\right)},
\end{equation}
define the interval
\begin{equation}\label{f,g domain}
    I_{\delta, X} := \left(\frac{6-3\delta}{1-9\delta}, \left(\frac{3-2\delta}{2(1+\delta)} \right)^X - \frac{1}{2} \right),
\end{equation}
which is non-empty as a consequence of the relations (\ref{parameter constraints}). With respect to these parameters, let also $f_{\theta, \delta}$ and  $g_{\theta,\delta,X}$ be functions defined on $ I_{\delta, X}$ given by
\begin{equation}\label{f function definition}
    f_{\theta,\delta} (t) := \frac{(t-5) \log(6 \sqrt{3}) + 6 \log 6 + (5t-1)\theta}{(1-9 \delta)t - (6-3\delta)},
\end{equation}
and
\begin{equation}\label{g function definition}
    g_{\theta, \delta, X} (t) := \frac{\log 2 + 4\theta}{(3-2\delta) - 2(1+\delta)(t+1/2)^{1/X}}\cdot
\end{equation}
Theorem \ref{main theorem soft} admits the following quantitative refinement.
\begin{theorem}\label{main theorem hard}
    Let $\bm \beta \in \Gm^2$ be non-torsion, and let $\bm F : \N_0 \to \Gm^2$ be a function for which there exists $\delta \in [0,1/9)$ and $\theta \geq 0$ such that
    \begin{equation*}
        \forall \; n \in \N_0: \quad d \big( \bm F(n), \bm \beta^n \big) \leq \theta + n \delta h(\bm \beta).
    \end{equation*}
    Then, for all
    \begin{equation*}
         X > \frac{\log \left( \frac{6-3\delta}{1-9\delta} + \frac{1}{2} \right)}{\log \left(\frac{3-2\delta}{2(1+\delta)}\right)},
    \end{equation*}
    it holds that
    \begin{equation*}
        \#\big( \bm F(\N_{\geq \nu}) \cap \Pi \big) \leq \lceil X \rceil,
    \end{equation*}
    with
    \begin{equation*}
        \nu = \frac{1}{h(\bm \beta)} \cdot \min_{t \in I_{\delta, X}} \max \big\{ f_{\theta, \delta}(t) ; g_{\theta, \delta, X}(t) \big\}.
    \end{equation*}
\end{theorem}
For some applications of this theorem, it is more convenient to work with a function $\bm F$ whose domain is a proper subset of $\N_0$, in which case the conclusion of the theorem still holds. Elementary considerations show that $f_{\theta,\delta}$ is strictly decreasing on $I_{\delta,X}$ and has a vertical asymptote at the left boundary, while $g_{\theta,\delta,X}$ exhibits the opposite behaviour: it is strictly increasing on $I_{\delta, X}$ and has a vertical asymptote at the right boundary. Hence the minimum
\begin{equation*}
    \min_{t \in I_{\delta, X}} \max \big\{ f_{\theta, \delta}(t) ; g_{\theta, \delta, X}(t) \big\}
\end{equation*}
is realised as the unique solution to the equation $f_{\theta, \delta}(t) = g_{\theta, \delta, X}(t)$ for $t \in I_{\delta, X}$, which may be determined by numerical methods. In the special case where $\delta = 0$ it is possible to avoid determining the intersection of the two graphs by instead choosing particular values of $t \in I_{\delta, X}$ which permit small values of the threshold $\nu$ and the counting bound $X$ respectively.

\begin{corollary}\label{corollary to main theorem}
    Let $\bm F$ and $\theta$ be as in Theorem \ref{main theorem hard} in the case where $\delta = 0$. Then
    \begin{equation*}
        \# \big( \bm F(\N_{\geq \nu_1}) \cap \Pi \big) \leq 24 \quad \operatorname{for} \quad \nu_1 = \frac{2.8 + 6 \theta}{h(\bm \beta)},
    \end{equation*}
    and
    \begin{equation*}
        \# \big( \bm F(\N_{\geq \nu_2}) \cap \Pi \big) \leq 5 \quad \operatorname{for} \quad \nu_2 = \frac{33.1 + 73.1 \theta}{h(\bm \beta)}\cdot
    \end{equation*}
\end{corollary}
\begin{proof}
    To obtain the first bound set $(X,t) = (24,35)$, yielding the values
    \begin{equation*}
        f_{\theta,0}(t) < 2.8 + 6\theta, \quad \text{and} \quad g_{\theta, 0, X}(t) < 1.1 + 5.9 \theta.
    \end{equation*}
    Then Theorem \ref{main theorem hard} yields the claim. Likewise, to retrieve the second bound, set 
    $(X,t) = (5,6.42591)$ to obtain
    \begin{equation*}
        f_{\theta,0}(t) < 33.1 + 73.1\theta, \quad \text{and} \quad g_{\theta, 0, X}(t) < 12.7 + 73.1 \theta,
    \end{equation*}
    from which another application of Theorem \ref{main theorem hard} yields the desired result.
\end{proof}

\subsection{Height and Arithmetic Distance in Projective Space}\label{projective notations}
Whereas the results recorded in the Introduction considered functions valued in $\Gm^2$, the results in this section concern functions taking values in
\begin{equation*}
    \mathbb{P}^\circ \big( \overline \Q^3 \big) := \left\{[x_1:x_2:x_3] \in \mathbb{P} \big( \overline \Q^3 \big): \; x_1x_2x_3 \neq 0\right\};
\end{equation*}
the elements of the projective space of  $\overline \Q^3$ whose homogeneous coordinates do not lie in any proper coordinate subspace of $\overline \Q^3$. To apply Theorem \ref{main theorem hard} in this context it is necessary to devise a notion of height and distance in the projective regime, complementary to those of (\ref{affine height definition}-\ref{affine distance definiton}). It is more convenient, however, to formulate this setup within $\Gm^3$ by defining height and distance functions which are invariant with respect to the action of the subgroup $\overline  \Q^*$ on $\Gm^3$. In particular, one defines a projective invariant height function on $\Gm^3$ via the formula
\begin{equation}
\forall \; \bm x \in \Gm^3: \quad h^*(\bm x) := \max_{1 \leq i < j \leq 3} h\left( \frac{x_i}{x_j}\right).
\end{equation}
Clearly, for all non-zero algebraic numbers $\lambda$ one has $h^*(\lambda \bm x) = h^*(\bm x)$. An equivalent definition for $h^*$ may be given in terms of the mappings (for each $1 \leq i \leq 3$)
\begin{align*}
    \varphi_i : \Gm^3 &\longrightarrow \Gm^2 \\
    \bm x &\longmapsto \left( \frac{x_j}{x_i} \right)_{\substack{1 \leq j \leq 3 \\ j \neq i}}
\end{align*}
which when viewed modulo $\overline \Q^*$ may be regarded as bijections between $\mathbb{P}^\circ \big( \overline \Q^3 \big)$ and $\Gm^2$, with $\mathbb{P}^\circ \big( \overline \Q^3 \big)$ suitably identified with $\Gm^3 \big/ \overline \Q^*$.  Then
\begin{equation*}
    \forall \;\bm x \in \Gm^3: \quad h^*(\bm x) = \max_{1 \leq i \leq 3} h \big(\varphi_i(\bm x)\big).
\end{equation*}
With respect to this projective invariant height, one may induce a distance function
\begin{equation*}
    \forall \; \bm x, \bm y \in \Gm^3: \quad d^*(\bm x, \bm y) := h^*(\bm x \bm y^{-1}).
\end{equation*}
It is clear too that this is invariant under rescaling by non-zero algebraic numbers.

\subsection{Pseudo-Effective Results at Order 3}\label{pseudo-effectivity at order 3}
Given the projective formalism described above, a refined version of Theorem \ref{Order 3 Theorem simplified} may now be stated, for which some additional definitions are required. First, with $\bm \alpha \in \Gm^s$, $s \in \N_{\geq 2}$ fixed, consider functions of the form
\begin{align*}
    f: \N_0 &\longrightarrow \overline \Q \\
    n &\longmapsto c_1 \alpha_1^n + ... + c_s \alpha_s^n.
\end{align*}
For such a function, let $\operatorname{coeff}_{\bm \alpha}(f)$ denote its coefficient vector $(c_1,...,c_s)$, which is assumed to lie in $\Gm^s$. For each $\theta \geq 0$, denote by $\L_s(\bm \alpha, \theta)$ the set of all functions of the above form, where the coefficient vector obeys
\begin{equation*}
    h^*\big(\operatorname{coeff}_{\bm \alpha}(f)\big) \leq \theta.
\end{equation*}
Two functions $f,g \in \L_s(\bm \alpha, \theta)$ are said to be \emph{equivalent} elements of $\L_s(\bm \alpha, \theta)$ if
\begin{equation*}
    \exists \; n \in \Z \quad \operatorname{s.t.} \quad \big[ \operatorname{coeff}_{\bm \alpha}(f) \big] = \big[ \bm \alpha^n \operatorname{coeff}_{\bm \alpha}(g) \big],
\end{equation*}
where $[\bm x]\in \mathbb{P}^\circ \big( \overline \Q^s \big)$ denotes the projective point determined by a given $\bm x \in \Gm^s$. 

It is clear that the zero sets of any two equivalent elements of $\L_s(\bm \alpha, \theta)$ are translates of one another, which justifies the stratification of $\L_s(\bm \alpha, \theta)$ according to the classes determined by this equivalence relation. Given an arbitrary function $f : \N_0 \to \overline \Q$ and an integer $\nu \in \N_0$, let
\begin{equation*}
    \operatorname{mult}_{\geq \nu} (f) := \# \big\{ n \geq \nu : f(n) = 0 \big\} \in \N_0 \cup \{\infty\}
\end{equation*}
denote the zero-multiplicity of $f$ beyond the threshold $\nu$. The following theorem provides small upper bounds for the sum of the zero-multiplicities of inequivalent elements of $\L_3(\bm \alpha, \theta)$, and clearly implies Theorem \ref{Order 3 Theorem simplified}.
\begin{theorem}\label{order 3 pseudo-effectivity theorem}
    Let $\theta \geq 0$ and let $\bm \alpha \in \Gm^3$ be such that $h^*(\bm \alpha) > 0$. Then, given any collection of pairwise inequivalent functions $f_1,...,f_N \in \L_3(\bm \alpha, \theta)$, one has
    \begin{equation*}
        \sum_{j=1}^N \operatorname{mult}_{\geq \nu_1} (f_j) \leq 24 \quad\operatorname{for} \quad \nu_1 = \frac{2.8 + 6\theta}{h^*(\bm \alpha)},
    \end{equation*}
    and
    \begin{equation*}
        \sum_{j=1}^N \operatorname{mult}_{\geq \nu_2} (f_j) \leq 5
        \quad\operatorname{for} \quad \nu_2 = \frac{33.1 + 73.1\theta}{h^*(\bm \alpha)}\cdot
    \end{equation*}
\end{theorem}

\begin{proof}
Let $\nu\in \{\nu_1, \nu_2\}$ and let
\begin{equation*}
    n_T > n_{T-1} > ... > n_1 \geq \nu
\end{equation*}
be integers such that for all $1 \leq i \leq T$ at least one of $f_1,...,f_N$ vanishes at $n_i$. It is claimed that for all $1 \leq i \leq T$ there exists a unique index $1 \leq j \leq N$ with $f_j(n_i) = 0$. Suppose this were not the case. Then there exists $1 \leq j < k \leq N$ with $f_j(n_i) = f_k(n_i) = 0$. Without loss of generality assume that $h(\alpha_1/\alpha_3) = h^*(\bm \alpha)$, and define
\begin{equation*}
\bm \beta := \varphi_3(\bm \alpha), \quad \bm b(f_l) := \varphi_3(-\operatorname{coeff}_{\bm \alpha}(f_l)) \quad ; \quad l \in \{j,k\}.
\end{equation*}
Then,
\begin{align*}
    b_1(f_j) \beta_1^{n_i} &+b_2(f_j) \beta_2^{n_i} = 1, \\
    b_1(f_k) \beta_1^{n_i} &+b_2(f_k) \beta_2^{n_i} = 1,
\end{align*}
which implies
\begin{equation*}
    \big( b_1(f_j) b_2(f_k) - b_1(f_k) b_2(f_j) \big)\beta_1^{n_i} = b_2(f_k) - b_2(f_j).
\end{equation*}
The simultaneous vanishing of both sides of this equation may be immediately ruled out, since it would imply that $\bm b(f_j) = \bm b(f_k)$, contradicting the inequivalence of $f_j$ and $f_k$. By using the standard height inequality recorded in Lemma \ref{heights lemma} of \S \ref{preliminary lemmas}, it follows that
\begin{align*}
    \nu h(\bm \beta) \leq n_i h( \bm \beta) &= h\left(\frac{b_2(f_k) - b_2(f_j)}{b_1(f_j) b_2(f_k) - b_1(f_k) b_2(f_j)} \right) \\
    &\leq \log 2 + 4\theta < \nu h(\bm \beta),
\end{align*}
which is a contradiction. This proves the claim that for all $1 \leq i \leq T$ there exists a \emph{unique} $1 \leq j(i) \leq N$ such that $f_{j(i)}$ vanishes at $n_i$. This in turn implies that
\begin{equation}\label{some relation}
    T = \sum_{j=1}^N \operatorname{mult}_{\geq \nu} (f_j).
\end{equation}
Now define the function
\begin{align*}
    \bm F : \{n_1,...,n_T\} &\longrightarrow \Gm^2 \\
    n_i &\longmapsto \bm b(f_{j(i)}) \bm \beta^{n_i}.
\end{align*}
It is further claimed that this function is injective. Otherwise there would exist indices $1 \leq i < i' \leq T$ such that $\bm F(n_i) = \bm F(n_{i'})$, implying
\begin{equation*}
    \bm b(f_{j(i)}) = \bm b(f_{j(i')}) \bm \beta^{n_{i'} - n_i}.
\end{equation*}
Since $n_i \neq n_{i'}$, and $\bm \beta$ is non-torsion, it cannot hold that $j(i) = j(i')$. Alternatively, if $j(i) \neq j(i')$ the above contradicts the inequivalence of $f_{j(i)}$ and $f_{j(i')}$. Hence, the injectivity of $\bm F$ is established, and in conjunction with Corollary \ref{corollary to main theorem} one has
\begin{equation*}
    T = \#\big( \bm F(\N_{\geq \nu}) \cap \Pi \big) \leq 24 \quad \operatorname{when}\quad \nu = \nu_1
\end{equation*}
and
\begin{equation*}
    T = \#\big( \bm F(\N_{\geq \nu}) \cap \Pi \big) \leq 5  \quad \operatorname{when}\quad \nu = \nu_2.
\end{equation*}
Finally, in view of the relation (\ref{some relation}) the proof is complete.
\end{proof}

\subsection{Pseudo-Effective Results at Higher Order Subject to Low Height Root Clustering}\label{pseudo-effectivity at higher orders}

Whereas in \S \ref{pseudo-effectivity at order 3} it was sufficient to apply Corollary \ref{corollary to main theorem} to conclude the proof, one may apply the fully general Theorem \ref{main theorem hard} to extend the pseudo-effective framework to higher order linear recurrences in the case where several characteristic roots lie close together with respect to the distance function $d$ on $\Gm \cong \overline \Q^*$. That is: to study solutions to equations of the form
\begin{equation*}
    v_1 \lambda_1^n + \dots + v_k \lambda_k^n = 0  \quad \quad\operatorname{given}\quad\quad v_1,\dots,v_k;\lambda_1,..., \lambda_k\in \overline \Q^*,
\end{equation*}
for which there exists a $\lambda \in \overline \Q^*$ and a sufficiently small $\varepsilon > 0$ such that
\begin{equation*}
    \forall\; 3 \leq i \leq k : \quad d(\lambda_i,\lambda) \leq \varepsilon.
\end{equation*}
A natural way to encode this setup in the language of the previous subsection is to first fix an $\bm \alpha \in \Gm^2$, with $h^*(\bm \alpha) > 0$, and a $\bm \xi \in \Gm^s$ such that $h^*(\bm \xi) \leq \varepsilon$, and then consider the equation
\begin{equation*}
    f(n) = g(n), \quad f \in \L_2(\bm \alpha, \theta),\; g \in \L_s(\bm \xi, \theta),
\end{equation*}
with $\theta \geq 0$ given. This yields the following generalisation of the pseudo-effective framework.
\begin{theorem}\label{higher order pseudo-effectivity theorem}
    Let $\theta \geq 0$, $s \in \N$, and let $\bm \alpha \in \Gm^2$, $\xi \in \Gm^s$ be such that $h^*(\bm \alpha) > 0$ and $h^*(\bm \xi) \leq \delta h^*(\bm \alpha)$ for some $\delta \in [0,1/9)$. Then, given any collection of pairwise inequivalent functions $f_1,...,f_N \in \L_2(\bm \alpha, \theta)$ and arbitrary functions $g_1,..., g_N \in \L_s(\bm \xi, \theta)$, if
    \begin{equation*}
         X > \frac{\log \left( \frac{6-3\delta}{1-9\delta} + \frac{1}{2} \right)}{\log \left(\frac{3-2\delta}{2(1+\delta)}\right)},
    \end{equation*}
    then
    \begin{equation*}
        \sum_{j=1}^N \operatorname{mult}_{\geq  \nu} (f_j - g_j) \leq \lceil X \rceil
    \end{equation*}
    for
    \begin{equation*}
        \nu = \frac{1}{h^*(\bm \alpha)} \cdot \min_{t \in I_{\delta, X}} \max \big\{ f_{\theta, \delta}(t) ; g_{\theta, \delta, X}(t) \big\}.
    \end{equation*}
\end{theorem}
The proof of this theorem proceeds nearly identically to that of Theorem \ref{order 3 pseudo-effectivity theorem}. One constructs a function $\bm F : S \to \Gm^2$, for some $S \subseteq \N_0$, to which Theorem \ref{main theorem hard} may be applied in place of Corollary \ref{corollary to main theorem}; its injectivity follows from the inequivalence of $f_1,\ldots,f_N$. For this reason, the proof of Theorem \ref{higher order pseudo-effectivity theorem} is omitted here.

To conclude this section, Theorem \ref{higher order pseudo-effectivity theorem} is applied to the example (\ref{high order example}) to derive the quantitative bounds given in the Introduction. The characteristic polynomial of the recurrence relation (\ref{high order example}) is given by
\begin{equation*}
    f(X) = A(X) B(X),
\end{equation*}
where
\begin{equation*}
    A(X) := X^2 + \frac{19}{10}X + 1, \quad B(X) := X^4 C(X + X^{-1}),
\end{equation*}
and
\begin{equation*}
     C(Y) := Y^4- \frac{1}{2} Y^3 - 2Y^2 + \frac{1}{2} Y + \frac{1}{2}\cdot
\end{equation*}
Since $C(Y)$ has all of its roots in the interval $(-2,2)$, it is easy to show that the roots of $B(X)$, as well as $A(X)$, all have modulus $1$, hence all ten characteristic roots of (\ref{high order example}) lie on the unit circle. The polynomial \(B(X)\) is irreducible over \(\mathbb Q\), whence by expanding \(B(X)\) and computing its Mahler measure, each of its roots has height \(\frac18\log2\). This gives an upper bound on the projective invariant height of $\bm \xi = (\xi_1,...,\xi_8)$:
\begin{equation*}
    h^*(\bm \xi) \leq  \max_{1 \leq i < j \leq 8} \big(h(\xi_i) + h(\xi_j) \big) \leq \frac{\log 2}{4}\cdot
\end{equation*}
Write $\alpha_1, \alpha_2$ for the roots of $A(X)$, and set $\bm \alpha := (\alpha_1, \alpha_2$), giving
\begin{equation*}
    h^*(\bm \alpha) = \log 10.
\end{equation*}
It is clear that the exponential polynomial representation of a sequence satisfying (\ref{high order example}) and (\ref{coefficient height order 10}) may be written
\begin{equation}
    a_n = f(n) - g(n) \quad \text{where} \quad f \in \mathcal{L}_2(\bm \alpha, \theta) \quad \text{and} \quad g\in \mathcal{L}_8(\bm \xi, \theta).
\end{equation}
Two such sequences $$\big(a_n = f(n) - g(n)\big)_{n \geq 0},\quad \big(a_n' = f'(n) - g'(n)\big)_{n \geq 0}$$ are regarded as equivalent if $f, f'$ are equivalent elements of $\mathcal{L}_2(\bm \alpha, \theta)$. Setting
\begin{equation*}
    \delta:= \frac{\log 2}{4 \log 10} \geq \frac{h^*(\bm \xi)}{h^*(\bm \alpha)} ,
\end{equation*}
one has
\begin{equation*}
    \delta = 0.07526...< 1/9,
\end{equation*}
and Theorem \ref{higher order pseudo-effectivity theorem} applies. Substituting the above $\delta$ and \(h^*(\alpha)=\log10\) into Theorem \ref{higher order pseudo-effectivity theorem}, and numerically solving \(f_{\theta,\delta}(t)=g_{\theta,\delta,X}(t)\) for $\theta = 10$ and $\theta = 1$, gives Tables \ref{tab:X-nu} and \ref{tab2:X-nu}.

\section{Linear Orbits Intersecting Expanding Arithmetic Neighbourhoods of Subspaces}\label{linear orbits}
In this brief section Theorem \ref{main theorem hard} is applied to a moving target problem in linear dynamics, generalising the following classical question: Given $\bm x \in (\overline{\Q}^*)^3$ and a diagonalisable matrix $M \in \operatorname{Mat}_{3\times 3}(\overline{\Q})$, determine the set of all $n \in \N_0$ for which
$$
M^n \bm x \in V,
$$
where $V$ is a fixed proper subspace of $(\overline{\Q}^*)^3$. This problem is a well-known reformulation of the Skolem Problem (for order three LRS's) in the language of linear dynamics. In particular, for matrices of dimension three, the set of such integers $n$ is effectively computable.

In this section one studies a quantitative generalisation of this question. Rather than considering exact intersections of the linear orbit with $V$, one instead considers those indices for which the orbit comes arithmetically close to $V$, with respect to the notions of arithmetic distance introduced in the previous sections. More precisely, one studies those integers $n$ for which the point $M^n\bm x$ lies within a prescribed arithmetic neighbourhood of $V$, where the size of this neighbourhood is allowed to vary with $n$. To formulate this precisely, some additional notation is first introduced, followed by the statement of Theorem \ref{theorem for linear orbits}.

In the following theorem, let $\operatorname{Diag}^*_3$ denote the set of all $3 \times 3$ diagonal matrices whose diagonal elements are non-zero algebraic numbers. Further, given $\bm \lambda \in \Gm^3$ write $\operatorname{diag}(\bm \lambda) \in \operatorname{Diag}_3^*$ for the matrix with diagonal $(\lambda_1, \lambda_2,\lambda_3)$. With this notation, the projective invariant height function $h^*$ may be applied to $\operatorname{Diag_3^*}$ via the formula
\begin{equation*}
    h^*\big(\operatorname{diag}(\bm \lambda)\big) := h^*(\bm \lambda).
\end{equation*}
Finally, given $\bm v \in \Gm^3$ write $\bm v^{\perp}$ for the set
\begin{equation*}
    \bm v^{\bm \perp} := \left\{\bm x \in \Gm^3 : v_1 x_1 + v_2 x_2 + v_3 x_3 = 0 \right\}.
\end{equation*}
\begin{theorem}\label{theorem for linear orbits}
Let $\bm v \in \Gm^3$ and let $D \in \operatorname{Diag}_3^*$ be such that $h^*(D) > 0$. Then, for all $\delta \in (0,1/9)$ and all $\bm x \in \Gm^3$ the inequality
\begin{equation}\label{dynamical inequality}
    d^*\big(D^n \bm x, \bm v^{\perp}\big) \leq n \delta h^*(D)
\end{equation}
is satisfied only finitely often in $n \in \N_0$.
\end{theorem}
\begin{proof}
Let $D = \operatorname{diag}(\bm \lambda)$ for some $\lambda \in \Gm^3$ and without loss of generality assume that $h(\lambda_1/\lambda_3) = h^*(\bm \lambda) > 0$. Set also
\begin{equation*}
    S = \left\{ n \in \N_0 :  d^*\big(D^n \bm x, \bm v^{\perp}\big) \leq n \delta h^*(D) \right\}.
\end{equation*}
To reformulate in $\Gm^2$, set
\begin{equation*}
    \bm \beta := \varphi_3(\bm \lambda), \quad \bm c := \varphi_3(-\bm x), \quad \bm u := \varphi_3(-\bm v).
\end{equation*}
Then $h(\bm \beta) = h(\beta_1)$, and for all $n \in S$ there exists a $\bm w(n) \in \Gm^2$ satisfying
\begin{equation*}
    d \big(\bm \beta^n \bm c, \bm w(n) \big) \leq n \delta h(\bm \beta) \quad \text{and} \quad u_1 w_1(n) + u_2 w_2(n) = 1.
\end{equation*}
By defining for each $n \in S$:
\begin{equation*}
    \bm \xi(n) := \bm w (n) \bm \beta^{-n} \bm c^{-1},
\end{equation*}
one has
\begin{equation*}
    h\big( \bm \xi(n)\big) \leq n \delta h(\bm \beta),
\end{equation*}
and
\begin{equation*}
    u_1 c_1 \xi_1(n) \beta_1^n + u_2 c_2 \xi_2(n) \beta_2^n = 1.
\end{equation*}
Therefore, for the function
\begin{align*}
    \bm F : S &\longrightarrow \Gm^2 \\
    n &\longmapsto \bm u \bm c \bm \xi(n) \bm \beta^n
\end{align*}
there exists a $\theta \geq 0$ such that
\begin{equation*}
    \forall n \in S: \quad d \big( \bm F(n), \bm \beta^n \big) \leq \theta + n \delta h(\bm \beta).
\end{equation*}
Noting that the image of $\bm F$ is contained in $\Pi$, Theorem \ref{main theorem hard} implies that $\operatorname{im}(\bm F)$ is a finite set. The proof is complete once it is shown that the preimage of every $\bm a \in \operatorname{im}(\bm F)$ is finite.

Given $\bm a \in \operatorname{im} (\bm F)$, let $n \in S$ be the least element of $\bm F^{-1}(\bm a)$, and let $m > n$ be any other element of $\bm F^{-1}(\bm a)$. By setting $r:= m-n > 0$, one obtains
\begin{equation*}
    \bm F(n) = \bm a = \bm F(m) \quad \implies \quad \bm \xi(n) = \bm \xi (m) \bm \beta^r.
\end{equation*}
Recalling the assumption that $h(\beta_1) = h(\bm \beta)$, it follows that
\begin{align*}
    r h(\bm \beta) &\leq h(\xi_1(n)) + h(\xi_1(m)) \\
    & \leq \delta h(\bm \beta) (2n + r).
\end{align*}
Therefore,
\begin{equation*}
    m \leq \left(1 + \frac{2\delta
    }{1-\delta}\right) n < \frac{5}{4}\cdot n,
\end{equation*}
where the inequality $\delta < 1/9$ is applied at the last step. Hence  $\bm F^{-1}(\bm a)$ is contained in a finite interval, which completes the proof.
\end{proof}
This theorem is not only non-effective---it does not offer any method for determining the finite set of $n \in \N_0$ solving (\ref{dynamical inequality})---but also non-quantitative. Indeed, as is implicit in the proof, the set of solutions lies in a union of intervals of the form $[n, (1 + \e) n]\cap \N_0$ with $n$ ranging over a finite set, and $\e = 2\delta/(1-\delta)$. Without bounds on how large this index $n$ may be, it is not possible to bound the cardinality of the set of solutions to (\ref{dynamical inequality}).
\section{Proof of Theorem \ref{main theorem hard}}\label{proof of the main theorem}
In this section Theorem \ref{main theorem hard} is proved. The main technical tool used to this end is Lemma \ref{Beukers Tijdemann lemma}, which reformulates \cite[Lemma 2.3]{BeukersEvertse1996} and \cite[Lemma 6]{beukersTijdemann1984}. This gives an upper bound for the gaps between the solutions to the equation (\ref{generalised equation}), recorded in Lemma \ref{upper bound on gaps} of \S \ref{section establishing upper bound}. Then in \S \ref{section establishing lower bound and completing proof} lower bounds on these gaps are obtained from elementary properties of the logarithmic height function, quickly leading to the completion of the proof. 

In preparation for the proof, some notation and data are fixed. First, one denotes by
\begin{equation*}
    \operatorname{round}(x)
\end{equation*}
the integer nearest to $x \in \R_{\geq 0}$, with the convention that half integers are rounded up: $\operatorname{round}(n +\frac{1}{2}) := n+1$ for all $n \in \N_0$. As a consequence $x < \operatorname{round}(x) + \frac{1}{2}$ for all $x \in \R_{\geq 0}$. Also, the standard notation
\begin{equation}
    \llbracket a, b \rrbracket := [a,b]\cap \Z, \quad \llparenthesis a, b \rrparenthesis := (a,b) \cap \Z
\end{equation}
and any combination of these barred brackets are used. Throughout this section fix the following data. One has
\begin{equation*}
    \delta \in \left( 0, \frac{1}{9} \right), \quad \theta \geq 0,
\end{equation*}
and a non-torsion $\bm \beta \in \Gm^2$. Further, one fixes a function
\begin{align*}
     \bm F : \N_0 &\longrightarrow \Gm^2 \\
     n &\longmapsto \big( F_1(n), F_2(n) \big)
\end{align*}
such that the associated function
\begin{align*}
     \psi : \N_0 &\longrightarrow \R_{\geq 0} \\
     n &\longmapsto h \big(\bm F(n) \beta^{-n}\big)
 \end{align*}
satisfies
\begin{equation*}
     \forall \; n \in \N_0: \quad\psi(n) \leq \theta + n \delta h(\bm \beta).
\end{equation*}
\subsection{Preliminary Lemmas}\label{preliminary lemmas}
As remarked above, the following lemma provides the main input for Lemma \ref{upper bound on gaps}.
\begin{lemma}[\cite{BeukersEvertse1996,beukersTijdemann1984}]\label{Beukers Tijdemann lemma}
    Let $u,v,\gamma \in \overline \Q^*$ and let $s \in \N$ be such that
    \begin{equation*}
        u \gamma^{2s}+v (1-\gamma)^{2s} = 1.
    \end{equation*}
    Then 
    \begin{equation*}
    h(\gamma) \leq \frac{1}{s} \log 2 + \log(6\sqrt{3}) + \frac{1}{s} \big( h(u) + h(v) \big).
    \end{equation*}
\end{lemma}
From this one quickly obtains the following
\begin{lemma}\label{fundamental lemma}
    Let $u,v,\gamma \in \overline \Q^*$ and let $k \in \N_{\geq 6}$ be such that
    \begin{equation*}
        u \gamma^k+v (1-\gamma)^k = 1.
    \end{equation*}
    Then
    \begin{equation*}
        h( \gamma) \leq \log(6 \sqrt{3}) + \frac{6 \log 6}{k-5} + \frac{2}{k-5}\big( h(u) + h(v) \big).
    \end{equation*}
\end{lemma}
\begin{proof}
Let
\begin{equation*}
    u':= 
\begin{cases}
    u\gamma : k \equiv 1 \mod 2 \\
    u : k \equiv 0 \mod 2
\end{cases}
,\; v' :=
\begin{cases}
    v\gamma : k \equiv 1 \mod 2 \\
    v : k \equiv 0 \mod 2
\end{cases}
, \; s:=
\begin{cases}
    \frac{k-1}{2} : k \equiv 1 \mod 2 \\
    \frac{k}{2} : k \equiv 0 \mod 2.
\end{cases}
\end{equation*}
Then by Lemma \ref{Beukers Tijdemann lemma},
\begin{align}
    h(\gamma) &\leq \frac{1}{s} \log 2 + \log(6 \sqrt{3}) + \frac{1}{s} \big( h(u') + h(v') \big) \\
    &\leq \frac{2}{k-1} \log 2 + \log(6 \sqrt{3}) + \frac{2}{k-1} \big(2h(\gamma) + h(u) + h(v) \big).
\end{align}
Equivalently,
\begin{equation*}
    \frac{k-5}{k-1} \cdot h(\gamma) \leq \log(6 \sqrt{3}) + \frac{2 \log 2}{k-1} + \frac{2}{k-1} \big( h(u) + h(v) \big).
\end{equation*}
Brief algebraic manipulation therefore yields the claim.
\end{proof}
The final preliminary lemma concerns a height inequality of general utility (cf. \cite[Chap. 14, Lemma 14.10]{Masser2016} where it is given in terms of the multiplicative Weil height \(H:=\exp h\)). Below, given a polynomial $P \in \Z[X_1,...,X_N]$ one writes $\operatorname{len}(P)$ for the sum of the absolute values of the coefficients of $P$.

\begin{lemma}\label{heights lemma}
Given $P,Q \in \Z[X_1,\dots,X_N]$ of degree at most $L_i \in \N_{\geq 0}$ in $X_i$ for each $1 \leq i \leq N$, and non-zero algebraic numbers $\xi_1, \dots, \xi_N$ with $Q(\xi_1, \dots, \xi_N)$ non-zero, one has
\begin{equation*}
    h\left(\frac{P(\xi_1,\dots,\xi_N)}{Q(\xi_1, \dots, \xi_N)} \right) \leq \log \max\big\{\operatorname{len}(P); \operatorname{len}(Q)\big\} + L_1 h(\xi_1) +\dots+ L_N h(\xi_N).
\end{equation*}
\end{lemma}
\subsection{Bounding above gaps between solutions}\label{section establishing upper bound}
Recalling the definition (\ref{f function definition}) of $f_{\theta, \delta}$ one finds via an elementary application of the quotient rule that
\begin{align*}
    f'_{\theta, \delta}(t) 
    &= -\frac{(\log(6\sqrt{3} + 5\theta)(6-3\delta)t - (5 \log (5\sqrt{3}) - 6\log 6 + \theta)(1-9\delta)}{\big[(1-9\delta)t - (6-3\delta) \big]^2} \\
    &< -\frac{2\log 6 - 2\log 3 + 4\theta}{\big[(1-9\delta)t - (6-3\delta) \big]^2} \\
    &< 0.
\end{align*}
Therefore $f_{\theta, \delta}$ is strictly decreasing.
\begin{lemma}\label{upper bound on gaps}
Let $t > (6-3\delta)/(1-9\delta)$ and let $m,n$ be integers satisfying
\begin{itemize}
    \item $m > n \geq f_{\theta, \delta}(t)/h(\bm \beta)$
    \item $\bm F(n) \in \Pi$ and $ \bm F(m) \in \Pi$.
\end{itemize}
Then,
\begin{equation*}
    \operatorname{round}\left(\frac{m}{n}\right) < t.
\end{equation*}
\end{lemma}
\begin{proof}
If $k:= \operatorname{round}(m/n)$ is in the integer interval $\llbracket 1, 6 \rrbracket$ then trivially $k < t$. Hence, one may assume that $k \geq 7$. Now decompose $m$ as
\begin{equation*}
    m = kn +l \quad ; \quad l \in \llbracket -n/2, n/2 \rrparenthesis.
\end{equation*}
Then with $\bm c(q) := \bm F(q) \bm \beta^{-q}$ for all $q \in \N_0$, one has the equations
\begin{align*}
    c_1(n) \beta_1^n &+ c_2(n) \beta_2^n = 1, \\
    c_1(m)\beta_1^l \cdot \beta_1^{nk} &+ c_2(m) \beta_2^l \cdot \beta_2^{nk} = 1.
\end{align*}
Substituting the first equation into the second and setting
\begin{equation*}
    u:= c_1(m) c_1(n)^{-k} \beta_1^l, \quad v:= c_2(m) c_2(n)^{-k} \beta_2^l, \quad \gamma:= c_1(n) \beta_1^n,
\end{equation*}
gives
\begin{equation*}
    u\gamma^k + v (1-\gamma)^k = 1.
\end{equation*}
Since $k > 6$, Lemma \ref{fundamental lemma} applies, and it therefore holds that
\begin{equation}
    h\big(c_1(n) \beta_1^n \big) \leq \log(6\sqrt{3}) + \frac{6 \log 6}{k-5} + \frac{2}{k-5}\big(2k \psi(n) + 2 \psi(m) + \frac{n}{2} h(\bm \beta) \big).
\end{equation}
Without loss of generality, assume $h(\beta_1) \geq h(\beta_2)$. Then,
\begin{equation*}
    nh(\bm \beta) = h\big(c_1(n)^{-1} \cdot c_1(n) \beta_1^n \big) \leq \psi(n) + h \big( c_1(n) \beta_1^n \big),
\end{equation*}
which implies
\begin{equation*}
    \frac{k-6}{k-5} \cdot nh(\bm \beta) \leq \log(6\sqrt{3}) + \frac{6 \log 6}{k-5} + \frac{5k - 5}{k-5}\cdot \psi(n) + \frac{4}{k-5}\cdot \psi(m). 
\end{equation*}
Since one has
\begin{equation*}
    m < \left( k + \frac{1}{2} \right) n,
\end{equation*}
applying the definition of $\psi$ and carrying out simple algebraic manipulations yields
\begin{equation*}
    nh(\bm\beta) < f_{\theta,\delta}(k).
\end{equation*}
Therefore, the inequalities
\begin{equation*}
    f_{\theta, \delta}(t) \leq n h(\bm \beta) < f_{\theta,\delta}(k)
\end{equation*}
hold. Finally, since $f_{\theta, \delta}$ is strictly decreasing, it follows that $k < t$.
\end{proof}
\subsection{Lower bounds on gaps, and completion of the proof}\label{section establishing lower bound and completing proof}
To complement Lemma \ref{upper bound on gaps}, a lower bound on gaps between solutions to $\bm F(n) \in \Pi$ is obtained using the basic property of heights given by Lemma \ref{heights lemma}.
\begin{lemma}\label{lower bound on gaps}
Let $m,n$ be natural numbers satisfying:
\begin{itemize}
    \item $m > n > \frac{\log2 + 4 \theta}{1-4 \delta}\cdot \frac{1}{h(\bm \beta)}$
    \item $\bm F(m) \neq \bm F(n)$
    \item $\bm F(m) \in \Pi$ and $\bm F(n) \in \Pi$.
\end{itemize}
Then,
\begin{equation*}
     m \geq \frac{3n}{2(1+\delta)} \left[1 - \frac{2\delta}{3} - \frac{\log 2 + 4\theta}{3 h(\bm \beta)n}\right].
\end{equation*}
\end{lemma}
\begin{proof}
With $\bm c(q) := \bm F(q)\bm \beta^{-q}$ for all $q \in \N_0$, let $r:= m-n >0$, such that
\begin{align*}
    c_1(n) \beta_1^n &+ c_2(n) \beta_2^n = 1, \\
    c_1(m) \beta_1^r \cdot \beta_1^n &+ c_2(m) \beta_2^r \cdot \beta_2^n = 1.
\end{align*}
Eliminating $\beta_2^n$ terms gives
\begin{equation*}
    \big( c_1(n) c_2(m) \beta_2^r - c_1(m) c_2(n) \beta_1^r \big) \cdot \beta_1^n = c_2(m)\beta_2^r - c_2(n).
\end{equation*}
Assuming that both sides of the equation vanish simultaneously implies that $c_2(n) = c_2(m) \beta_2^r$, which quickly implies that $c_1(n) = c_1(m) \beta_1^r$, and therefore
\begin{equation*}
\bm F(n) = \bm c(n) \bm \beta^n = \bm c(m) \bm \beta^r \bm \beta^n = \bm c(m) \bm \beta^m = \bm F(m).
\end{equation*}
Therefore neither side of the above equation vanishes identically, and it follows that
\begin{align*}
    n h(\bm \beta) &= h \left(\frac{c_2(m) \beta_2^r - c_2(n)}{c_1(n)c_2(m) \beta_2^r - c_1(m) c_2(n)\beta_1^r} \right) \\
    &\leq \log 2 + 2 \psi(n) + 2 \psi(m) +2 rh(\bm \beta) \\
    &\leq \log2 + 4\theta +4\delta h(\bm \beta) n + 2(1+\delta) h(\bm \beta) r,
\end{align*}
where Lemma \ref{heights lemma} has been applied between the first and the second line of the above, and the definition of $\psi$ has been applied between the second and third line. Rearranging this inequality and recalling $m = n + r$ yields the stated inequality that was to be proved.
\end{proof}
Theorem \ref{main theorem hard} now follows almost immediately from Lemmas \ref{upper bound on gaps} and \ref{lower bound on gaps}.
\begin{proof}[Proof of Theorem \ref{main theorem hard}]
Let $t,s,X$ be real numbers satisfying
\begin{equation*}
    t \in I_{\delta,X} ,\quad  \quad s > \frac{\log 2 + 4\theta}{1-4\delta}, \quad X > \frac{\log \left( \frac{6-3\delta}{1-9\delta} + \frac{1}{2} \right)}{\log \left(\frac{3-2\delta}{2(1+\delta)}\right)},
\end{equation*}
and further set
\begin{equation*}
    \nu := \max \big\{ f_{\theta, \delta}(t) ; s \big\}/h(\bm \beta).
\end{equation*}
Now suppose there exist integers
\begin{equation*}
    n_T > n_{T-1} > ... > n_1 \geq \nu
\end{equation*}
such that
\begin{itemize}
    \item $\bm F(n_i) \neq \bm F(n_j)$ for all $1 \leq i < j \leq T$
    \item $\bm F(n_i) \in \Pi$ for all $1 \leq i \leq T$.
\end{itemize}
Then, on the one hand, Lemma \ref{upper bound on gaps} implies
\begin{equation*}
     \operatorname{round}\left(\frac{n_T}{n_1} \right) \leq t \quad \implies \quad n_T < \left( t + \frac{1}{2} \right) n_1,
\end{equation*}
and on the other hand, Lemma \ref{lower bound on gaps} implies
\begin{equation*}
    \forall \; 1 \leq i < T : \quad n_{i+1} \geq \frac{3n_i}{2(1+\delta)} \left[1 - \frac{2\delta}{3} - \frac{\log 2 + 4\theta}{3 s} \right].
\end{equation*}
Therefore,
\begin{equation*}
    n_1 \cdot \left[ \frac{3}{2(1+\delta)} \left(1 - \frac{2\delta}{3} - \frac{\log 2 + 4\theta}{3s} \right)\right]^{T-1} \leq n_T < \left(t + \frac{1}{2}\right) \cdot n_1,
\end{equation*}
yielding
\begin{equation*}
    T < 1 + \frac{\log \left(t +\frac{1}{2}\right)}{\log \left[ \frac{3}{2(1+\delta)} \left(1 - \frac{2\delta}{3} - \frac{\log 2 + 4\theta}{3 s} \right)\right]} \cdot
\end{equation*}
Finally, by setting
\begin{equation*}
    s = g_{\theta, \delta, X}(t),
\end{equation*}
it quickly follows from the definition (\ref{g function definition}) of $g_{\theta,\delta, X}$ that $T < 1 + X$, and one further has
\begin{equation*}
    \nu = \max \big\{ f_{\theta,\delta}(t); g_{\theta, \delta, X}(t)\big\}.
\end{equation*}
By taking the minimum of this threshold over all $t \in I_{\delta, X}$, and observing that $T$ is an integer, one obtains
\begin{equation*}
    T \leq \lceil X \rceil \quad \text{and} \quad \nu = \min_{t \in I_{\delta, X}} \max \big\{ f_{\theta,\delta}(t); g_{\theta, \delta, X}(t)\big\}.
\end{equation*}
This concludes the proof.
\end{proof}
\section*{References}
\bibliographystyle{siam}
\renewcommand\refname{\vspace{-6ex}}

\begingroup
\footnotesize
\bibliography{ref}
\endgroup

\end{document}